\documentclass[a4paper,11pt]{article}
\usepackage[french,english]{babel}
\usepackage{setspace}
\usepackage{diagrams}
\usepackage{geometry}
\usepackage{epsfig}
\usepackage{amsfonts}
\usepackage{amsthm}
\usepackage{latexsym}
\usepackage{amsmath}
\usepackage{amscd}
\usepackage{amssymb}
\usepackage{enumerate}
\usepackage{graphicx,oldgerm}

\usepackage{authblk}

\theoremstyle{definition}
\newtheorem{thm}{Theorem}[section]

\newtheorem{prop}[thm]{Proposition}
\newtheorem{defi}[thm]{Definition}
\newtheorem{rem}[thm]{Remark}
\newtheorem{note}[thm]{Notation}
\newtheorem{para}[thm]{}

\DeclareMathOperator{\p3}{\mathbb{P}^3}
\DeclareMathOperator{\pn}{\mathbb{P}^{2n+1}}

\DeclareMathOperator{\N}{\mathcal{N}}
\DeclareMathOperator{\T}{\mathcal{T}}

\DeclareMathOperator{\I}{\mathcal{I}}
\DeclareMathOperator{\mo}{\mathcal{O}}

\newcommand{\mb}[1]{\mathbb{#1}}
\newcommand{\mc}[1]{\mathcal{#1}}

\title{Semi-regular varieties and variational Hodge conjecture}
\author{Ananyo Dan \thanks{The author has been supported by the DFG under Grant KL-$2244/2-1$} \, \,  Inder Kaur \thanks{The author has been supported by Berlin Mathematical School
\\Mathematics Subject Classification($2010$): $14$C$30$, $14$D$07$, $32$G$20$, $13$D$10$\\Keywords: Noether-Lefschetz locus, Hodge theory, hypersurfaces, algebraic cycles, semi-regularity, variational Hodge conjecture, deformation theory}}
\date{\today}

\begin{document}
\maketitle

\selectlanguage{french} 
\begin{abstract}
D'apr\`{e}s \cite{b1,fl} nous savons que sous-vari\'{e}t\'{e}s semi-r\'{e}guli\`{e}rs satisfaisent la conjecture de Hodge variationnelle, c'est-\`{a}-dire, donn\'{e} une famille de vari\'{e}t\'{e}s projectives, lisses $ \pi: \mc{X} \to B $, une fibre sp\'{e}ciale $ \mc{X}_o$ et un semi-r\'{e}guli\`{e}re
  sous-vari\'{e}t\'{e} $ Z \subset \mc{X}_o$, la classe de cohomologie correspondant \`{a} $ Z $ reste une classe Hodge (comme $ \mc{X}_o $ d\'{e}forme le long $B$) si et seulement si $ Z $ reste un le cycle alg\'{e}brique. Dans cet article, nous \'{e}tudions des exemples de tels sous-vari\'{e}t\'{e}s.
  En particulier, nous prouvons que toute lisse vari\'{e}t\'{e} projective $ Z $ de dimension $ n $ est une sous-vari\'{e}t\'{e} semi-r\'{e}guli\`{e}re d'une hypersurface projective lisse dans  $ \mb{P}^{2n + 1} $ du grand degr\'{e} suffisant.
\end{abstract}
\selectlanguage{english} 
\begin{abstract}
 Following \cite{b1, fl} we know that semi-regular sub-varieties satisfy the variational Hodge conjecture i.e., given a family of smooth projective varieties $\pi:\mc{X} \to B$, a special fiber $\mc{X}_o$ and a semi-regular 
 subvariety $Z \subset \mc{X}_o$, the cohomology class corresponding to $Z$ remains a Hodge class (as $\mc{X}_o$ deforms along $B$) if and only if $Z$ remains an algebraic cycle. In this article, we investigate examples of such sub-varieties.
 In particular, we prove that any smooth projective variety $Z$ of dimension $n$ is a semi-regular sub-variety of a smooth projective hypersurface in $\mb{P}^{2n+1}$ of large enough degree.
\end{abstract}

\section{Introduction}
The aim of this article is to study examples of semi-regular varieties.
The semi-regularity for a curve on a surface was first introduced by \cite{sev}. This was later generalized
to arbitrary divisors on a complex manifold by Kodaira-Spencer in \cite{kod1}. In \cite{b1}, Bloch  extended the notion
to cycles corresponding to  local complete intersection subschemes. This was further generalized
by Buchweitz and Flenner in \cite{fl}.

One of the motivations for the study of semi-regular varieties comes from the variational Hodge conjecture, namely
these varieties satisfy the variational Hodge conjecture. In particular, Bloch in \cite{b1} and Buchweitz and Flenner in \cite{fl}
noticed that for a smooth projective variety $X$ and a semi-regular local complete intersection subscheme $Z$
in $X$, any infinitesimal deformation of $X$ lifts the cohomology class of $Z$ (which is a Hodge class)
to a Hodge class if and only if $Z$ lifts to a local complete intersection subscheme (in the deformed scheme).

In the case of a smooth hypersurface $X$ in $\p3$, an effective divisor $C$ in $X$ is said to be \emph{semi-regular}
if $h^1(\mo_X(C))=0$. If $C$ is smooth and $\deg(X)>\deg(C)+4$ then Serre duality implies that $h^1(\mo_X(C))=h^1(\mo_X(-C)(d-4))$ which is equal to zero because the Castelnuovo-Mumford regularity of $C$ is at most $\deg(C)$.
 But the description of the semi-regularity for subschemes which are not divisors is more complicated, as we see below in \S \ref{sem4}. The main result of this article generalizes the above case of divisors
 to higher codimension subvarieties (see \S \ref{sem7}). In particular, we prove

\begin{thm}\label{sem1}
 Let $Z$ be a smooth subscheme in $\pn$ of codimension $n+1$. Then for $d \gg 0$, there exists a smooth degree
 $d$ hypersurface in $\pn$ containing $Z$ such that $Z$ is semi-regular in $X$.
\end{thm}
 We finally observe in Remark \ref{sem3} that for such a choice of $Z$ and $X$, the cohomology class of $Z$ in $H^{n,n}(X,\mb{Z})$ satisfies the variational Hodge conjecture for a family of degree $d$ hypersurfaces in $\pn$
 with a special fiber $X$.

\section{Bloch's Semi-regularity map}\label{sem4}

 \begin{para}\label{ph08}
  In \cite{b1}, Bloch generalizes the above definition of semi-regularity for divisors to any local complete intersection subscheme in a smooth projective variety over an algebraically closed field. We briefly recall the definition.
  Let $X$ be a smooth  projective variety of dimension $n$ and $Z$ be a local complete intersection subscheme in $X$ of codimension $q$. 
 Consider the composition morphism \[\Omega_X^{n-q+1} \times \bigwedge^{q-1} \N_{Z|X}^\vee \xrightarrow{1 \times \bigwedge^{q-1}\bar{d}} \Omega_X^{n-q+1} \times \Omega_X^{q-1} \otimes \mo_Z \xrightarrow{\bigwedge} K_X \otimes \mo_Z\]
  where \[\bar{d}:\N_{Z|X}^\vee \cong \I_{Z|X}/\I_{Z|X}^2 \to \Omega^1_X \otimes \mo_Z\]is the map induced by the differential $d:\I_{Z|X} \to \Omega^1_X$, with $\I_{Z|X}$ denoting the ideal sheaf of $Z$ in $X$.
  By adjunction, this induces a map, \[\Omega_X^{n-q+1} \to \bigwedge^{q-1}\N_{Z|X} \otimes K_X \cong \N_{Z|X}^\vee \otimes K_Z^0,\]where $K_Z^0:=\bigwedge^{q}\N_{Z|X} \otimes K_X$ is the \emph{dualizing sheaf}.
  Dualizing the induced map in cohomology, 
  \[H^{n-q-1}(X,\Omega^{n-q+1}) \to H^{n-q-1}(Z,\N_{Z|X}^\vee \otimes K_Z^0), \mbox{ gives us } \pi:H^1(\N_{Z|X}) \to H^{q+1}(X,\Omega_X^{q-1}).\]
  \end{para}

  \begin{defi}
    The map $\pi$ is called the \emph{semi-regularity map} and if it is injective we say that $Z$ is \emph{semi-regular}.
  \end{defi}

  \section{Proof of Theorem \ref{sem1} and an application}\label{sem7}
  
\begin{para}
 Before we come to the final result of this article we recall a result by Kleiman and Altman which tells us given a smooth subscheme in $\pn$ of codimension $n+1$ there exist a \emph{smooth} 
 hypersurface in $\pn$ containing it.
\end{para}

\begin{note}
Let $Z$ be a projective subscheme in $\pn$. Denote by \[Z_e:=\{z \in Z| \dim \Omega^1_{Z,z}=e\}.\]
 \begin{thm}[{\cite[Theorem $7$]{kleim}}]\label{exi1}
  If for any $e>0$ such that $Z_e \not= \emptyset$ we have that $\dim Z_e+e$ is less than $2n+1$ then there exists a smooth hypersurface in $\pn$ 
  containing $Z$. Moreover, if $Z$ is $d-1$-regular (in the sense of Castelnuovo-Mumford) then
 there exists a smooth degree $d$ such hypersurface containing $Z$.
 \end{thm}
\end{note}

 We need the following proposition:
 \begin{prop}\label{sem2}
  Let $Z$ be a smooth subscheme in $\pn$ of codimension $n+1$ and $X$ be a smooth degree $d$ hypersurface in $\pn$
  containing $Z$ for some $d \gg 0$. Then, for any integers $2 \le i <n$, $h^n\left(\bigwedge^{i-1}\T_Z \otimes \bigwedge^{n-i}\N_{Z|X}(d-4)\right)=0$.
 \end{prop}

 \begin{proof}
  Since $X$ is a hypersurface in $\pn$, $\N_{X|\pn}$ is isomorphic to $\mo_X(d)$.  
 Under this identification, we get the following normal short exact sequence, \[0 \to \N_{Z|X} \to \N_{Z|\pn} \to \mo_Z(d) \to 0.\]
 This gives rise to the following short exact sequence for $0 \le i \le n$:
  \[0 \to \bigwedge^{n-i}N_{Z|X} \to \bigwedge^{n-i}\N_{Z|\pn} \to \left( \bigwedge^{n-i-1}\N_{Z|X} \right) \otimes \mo_Z(d) \to 0.\] 
 Denote by $\mc{F}_{j,k}:=\bigwedge^{j}\T_Z \otimes \mo_X(k)$ for some $j,k \in \mb{Z}_{\ge 0}$. Since $Z$ and $X$ are smooth, $\mc{F}_{j,k}$ is $\mo_Z$-locally free and hence $\mo_Z$-flat. Tensoring the previous short exact sequence by $\mc{F}_{j,k}$ then
 gives us the following short exact sequence, 
 \[0 \to \mc{F}_{j,k} \otimes \bigwedge^{n-i}\N_{Z|X} \to \mc{F}_{j,k} \otimes \bigwedge^{n-i}\N_{Z|\pn} \to \mc{F}_{j,k} \otimes \bigwedge^{n-i-1}\N_{Z|X}(d) \to 0.\]
 By Serre's vanishing theorem, for $d \gg 0, l>0$ and $m \ge 1$, $H^m\left(\mc{F}_{j,ld-4} \otimes \bigwedge^{n-i}\N_{Z|\pn}\right)=0$, hence 
 \begin{equation}\label{sem5}
  H^m\left(\mc{F}_{j,ld-4} \otimes \bigwedge^{n-i-1}\N_{Z|X}(d)\right) \cong H^{m+1}\left(\mc{F}_{j,ld-4} \otimes \bigwedge^{n-i}\N_{Z|X}\right).
 \end{equation}
 Using Serre's vanishing theorem again for $d \gg 0$ and $i \ge 1$, $h^i\left(\bigwedge^{i-1}\T_Z((n-i+1)d-4)\right)=0$. Hence, 
 using the isomorphism (\ref{sem5}) recursively, we get for $j=i-1$, \[h^n\left(\bigwedge^{i-1} \T_Z \otimes \bigwedge^{n-i}\N_{Z|X}(d-4)\right)=h^{n-1}\left(\bigwedge^{i-1} \T_Z \otimes \bigwedge^{n-i-1}\N_{Z|X}(2d-4)\right)=...\]\[...=h^i\left(\bigwedge^{i-1} \T_Z((n-i+1)d-4)\right)=0.\]
 This proves the proposition.
 \end{proof}

\begin{proof}[Proof of Theorem \ref{sem1}]
 The existence of a smooth hypersurface in $\pn$ containing $Z$ for $d \gg 0$ follows from Theorem \ref{exi1}.
 It suffices to prove that there exists a 
 hypersurface $X$ in $\pn$ of degree $d \gg 0$ containing $Z$ such that
 the morphism from $H^{n-1}(\Omega_X^{n+1} \otimes \mo_Z)$ to $H^{n-1}(\N_{Z|X}^{\vee} \otimes \bigwedge^n \N_{Z|X} \otimes K_X)$, which is the dual to the semi-regularity map $\pi$ (see \ref{ph08}),
 is surjective. 
 
 Consider the short exact sequence, \[0 \to \T_Z \to \T_X \otimes \mo_Z \to \N_{Z|X} \to 0.\]
 Consider the associated filtration, 
 \[0=F^n \subset F^{n-1} \subset ... \subset F^0=\bigwedge^{n-1}(\T_X \otimes \mo_Z) \mbox{ satisfying }
  F^p/F^{p+1} \cong \bigwedge^p \T_Z \otimes \bigwedge^{n-1-p} \N_{Z|X}
 \]
 for all $p$.
 Taking $p=0$ we get the following short exact sequence
 \[0 \to F^1 \to \bigwedge^{n-1}(\T_X \otimes \mo_Z) \to \bigwedge^{n-1}\N_{Z|X} \to 0.\]
 Tensoring this by $K_X$ and looking at the associated long exact sequence, we get 
 \[ ...\to H^{n-1}(\Omega_X^{n+1} \otimes \mo_Z) \to H^{n-1}(\N_{Z|X}^{\vee} \otimes \bigwedge^n \N_{Z|X} \otimes K_X) \to H^n(F^1(d-4)) \to ...\]
 It therefore suffices to prove that $h^n(F^1(d-4))=0$.  
 
 We claim that it is sufficient to prove $h^n(F^{n-1}(d-4))=0$. Indeed, suppose $h^n(F^{n-1}(d-4))=0$.
 By Proposition \ref{sem2}, for any integer $2 \le i \le n-1$, we have \[ h^n\left(\bigwedge^{i-1} \T_Z \otimes \bigwedge^{n-i} \N_{Z|X}(d-4)\right)=0.\]
 Consider the following short exact sequence, where $2 \le p \le n-1$,  
 \begin{equation}\label{sem8}
0 \to F^p \to F^{p-1} \to \bigwedge^{p-1} \T_Z \otimes \bigwedge^{n-p} \N_{Z|X} \to 0
 \end{equation}
 Tensoring (\ref{sem8}) by $K_X \cong \mo_X(d-4)$ and considering the corresponding long exact sequence, we can conclude $h^n(F^{n-2}(d-4))=0$ (substitute $p=n-1$).
Recursively substituting  $p=n-2, n-3,...,2$ in (\ref{sem8}), we observe that $h^n(F^i(d-4))=0$ for $i=1,...,n-2$. In particular $h^n(F^1(d-4))=0$. Hence, it suffices to
prove $h^n(F^{n-1}(d-4))=0$.
  
 Note that, $F^{n-1} \cong \bigwedge^{n-1}\T_Z$ does not depend on the choice of $X$, hence independent of $d$. Therefore, by Serre's vanishing theorem, $h^n(F^{n-1}(d-4))=0$ for $d \gg 0$.
  This completes the proof of the theorem.
\end{proof}

\begin{rem}\label{sem3}
 Notations as in Theorem \ref{sem1}.
 We now note that the theorem implies a very special case of the variational Hodge conjecture. Indeed, consider a family $\pi:\mc{X} \to S$ of smooth degree $d$ hypersurfaces in $\mb{P}^{2n+1}$ with $X$ as a special fiber.
 Denote by $\gamma$ the cohomology class of $Z$ in $X$. Then, using \cite[Theorem $7.1$]{b1} notice that $\gamma$ remains a Hodge class if and only if $Z$ remains an algebraic variety as $X$ deforms along $S$.
\end{rem}

\bibliographystyle{alpha}
 \bibliography{researchbib}

\begin{thebibliography}{Sev44}

\bibitem[BF03]{fl}
R.-O. Buchweitz and H.~Flenner.
\newblock A semiregularity map for modules and applications to deformations.
\newblock {\em Compositio Mathematica}, 137(2):135--210, 2003.

\bibitem[Blo72]{b1}
S.~Bloch.
\newblock Semi-regularity and de-{R}ham cohomology.
\newblock {\em Inventiones Math.}, 17:51--66, 1972.

\bibitem[KA79]{kleim}
S.~L. Kleiman and A.~B. Altman.
\newblock Bertini's theorem for hypersurface sections containing a subscheme.
\newblock {\em Communications in Algebra}, 7(8):775--790, 1979.

\bibitem[KS59]{kod1}
K.~Kodaira and D.~C. Spencer.
\newblock A theorem of completeness of characteristic systems of complete
  continuous systems.
\newblock {\em Amer. J. Math.}, 81(2):477--500, 1959.

\bibitem[Sev44]{sev}
F.~Severi.
\newblock Sul teorema fondamentale dei sistemi continui di curve sopra una
  superficie algebrica.
\newblock {\em Ann. Mat. Pura Appl.}, 23(4):149--181, 1944.

\end{thebibliography}

  \vspace{2cm}
   Humboldt Universit\"{a}t Zu Berlin, Institut f\"{u}r Mathematik, Unter den Linden $6$, Berlin $10099$, Germany,\\ E-mail address: dan@mathematik.hu-berlin.de\\
   
   Freie Universität Berlin, FB Mathematik und Informatik, Arnimallee 3, 14195 Berlin,
Germany.\\
E-mail address: kaur@math.fu-berlin.de
 \end{document}